\documentclass[11pt]{article}
\pdfoutput=1

\usepackage[margin=1.3in]{geometry}

\usepackage{graphicx}
\usepackage{amssymb,amsmath}
\usepackage{amsthm,enumerate}
\usepackage{hyperref,xcolor}
\usepackage{mathrsfs}
\usepackage{extarrows}
\usepackage{textcomp}

\makeatletter

\newdimen\bibspace
\makeatother

\numberwithin{equation}{section}

\newtheorem{theorem}{Theorem}[section]
\newtheorem{lemma}[theorem]{Lemma}
\newtheorem{proposition}[theorem]{Proposition}

\newtheorem{remark}[theorem]{Remark}

\def\<{\langle}
\def\>{\rangle}

\def\XXint#1#2#3{{\setbox0=\hbox{$#1{#2#3}{\int}$ }
\vcenter{\hbox{$#2#3$ }}\kern-.6\wd0}}

\begin{document}

\title{Asymptotic expansion and optimal symmetry of minimal gradient graph equations in dimension 2}

\author{Zixiao Liu,\quad Jiguang Bao\footnote{Supported in part by National Natural Science Foundation of China (11871102 and 11631002).}}
\date{\today}

\maketitle

\begin{abstract}
In this paper, we study asymptotic expansion at infinity and symmetry of zero mean curvature equations of gradient graph in dimension 2, which include the Monge--Amp\`ere equation, inverse harmonic Hessian equation and the special Lagrangian equation. This refines the research of asymptotic behavior, gives the precise gap between exterior minimal gradient graph and the entire case, and extends the classification results of Monge--Amp\`ere equations.
 \\[1mm]
 \\[1mm]
 {\textbf{MSC~2020:}}\quad 35J60; 35C20; 35B06
\end{abstract}

\section{Introduction}

In this paper, we study qualitative properties of classical solutions of
\begin{equation}\label{equ:Perturbed-General}
  F_{\tau}\left(\lambda\left(D^{2} u\right)\right)=C_0
\end{equation}
in exterior domain of $\mathbb R^2$ and the punctured space $\mathbb R^2\setminus\{0\}$,
where
$C_0$ is a constant, $\lambda\left(D^{2} u\right)=\left(\lambda_{1}, \lambda_{2}\right)$ is the vector formed by the eigenvalues of Hessian matrix $D^{2} u$,
$\tau\in [0,\frac{\pi}{2}]$ and
$$
F_{\tau}(\lambda):=\left\{
\begin{array}{ccc}
\displaystyle  \frac{1}{2} \sum_{i=1}^{2} \ln \lambda_{i}, & \tau=0,\\
\displaystyle  \frac{\sqrt{a^{2}+1}}{2 b} \sum_{i=1}^{2} \ln \frac{\lambda_{i}+a-b}{\lambda_{i}+a+b},
  & 0<\tau<\frac{\pi}{4},\\
  \displaystyle-\sqrt{2} \sum_{i=1}^{2} \frac{1}{1+\lambda_{i}}, & \tau=\frac{\pi}{4},\\
  \displaystyle\frac{\sqrt{a^{2}+1}}{b} \sum_{i=1}^{2} \arctan \displaystyle\frac{\lambda_{i}+a-b}{\lambda_{i}+a+b}, &
  \frac{\pi}{4}<\tau<\frac{\pi}{2},\\
  \displaystyle\sum_{i=1}^{2} \arctan \lambda_{i}, & \tau=\frac{\pi}{2},\\
\end{array}
\right.
$$
$a=\cot \tau, b=\sqrt{\left|\cot ^{2} \tau-1\right|}$.
Eq.~\eqref{equ:Perturbed-General} origins from zero mean curvature equation of gradient graph $(x, D u(x))$ in $(\mathbb{R}^{2} \times \mathbb{R}^{2},g_{\tau})$ where
\begin{equation*}
g_{\tau}=\sin \tau \delta_{0}+\cos \tau g_{0}, \quad \tau \in\left[0, \frac{\pi}{2}\right]
\end{equation*}
is the linearly combined metric of standard Euclidean metric
\begin{equation*}
\delta_{0}=\sum_{i=1}^{2} d x_{i} \otimes d x_{i}+\sum_{j=1}^{2} d y_{j} \otimes d y_{j}
\end{equation*}
and the pseudo-Euclidean metric
\begin{equation*}
g_{0}=\sum_{i=1}^{2} d x_{i} \otimes d y_{i}+\sum_{j=1}^{2} d y_{j} \otimes d x_{j}.
\end{equation*}

From the definition of $F_{\tau}$ operator, $\lambda_i$ must satisfy
$$
\left\{
\begin{array}{llll}
  \lambda_i>0, & \text{for }\tau=0,\\
  \frac{\lambda_i+a-b}{\lambda_i+a+b}>0, & \text{for }\tau\in(0,\frac{\pi}{4}),\\
  \lambda_i\not=-1, & \text{for }\tau=\frac{\pi}{4},\\
  \lambda_i+a+b\not=0, & \text{for }\tau\in(\frac{\pi}{4},\frac{\pi}{2}),\\
\end{array}
\right.
$$
for $i=1,2$. Thus we separate the solution into semi-convex and semi-concave cases.
We consider Eq.~\eqref{equ:Perturbed-General} under the following conditions
\begin{equation}\label{equ:condition-semiConvex}
  \left\{
  \begin{array}{lllll}
    D^2u>0, & \tau=0,\\
    D^2u>(-a+b)I, & \tau\in (0,\frac{\pi}{4}),\\
    D^2u>-I, & \tau=\frac{\pi}{4},\\
    D^2u>-(a+b)I, & \tau\in (\frac{\pi}{4},\frac{\pi}{2}),\\
    C_0\not=0, & \tau=\frac{\pi}{2},\\
  \end{array}
  \right.
\end{equation}
where $I$ denotes the 2-by-2 identity matrix and the semi-concave case can be treated similarly.

For $\tau=0$, \eqref{equ:Perturbed-General} becomes the Monge--Amp\`ere equation $$
\det D^2u=e^{2C_0}.
$$
J\"orgens \cite{Jorgens} proved that any classical solution of such Monge--Amp\`ere equation must be a quadratic polynomial. For Monge--Amp\`ere equation on exterior domain, Ferrer--Mart\'{\i}nez--Mil\'{a}n \cite{FMM99} proved that any convex solution must be asymptotic to quadratic polynomial with additional $\ln$-term near infinity. For different proofs and extensions, we refer to Cheng--Yau \cite{ChengandYau}, Caffarelli \cite{7}, Jost--Xin \cite{JostandXin}, Caffarelli--Li \cite{CL} etc.

For $\tau=\frac{\pi}{2}$, \eqref{equ:Perturbed-General} becomes the special Lagrangian equation
$$
\sum_{i=1}^2\arctan \lambda_i(D^2u)=C_0.
$$
Fu \cite{Leifu-Bernstein} proved an analogue of Bernstein's theorem, which states that any classical solution of such special Lagrangian equation is either harmonic or a quadratic polynomial. For special Lagrangian equation on exterior domain, Li-Li-Yuan \cite{Li-Li-Yuan-Bernstein-SPL} proved that any classical solution is either harmonic or asymptotic to quadratic polynomial with additional $\ln$-term at infinity. For further extensions on higher dimension case, we refer to Yuan \cite{Yu.Yuan1,Yu.Yuan2} etc.

For $\tau=\frac{\pi}{4}$, \eqref{equ:Perturbed-General} becomes the translated inverse harmonic Hessian equation
$$
\sum_{i=1}^2\dfrac{1}{\lambda_i(D^2u)}=1,
$$
which is a special form of Hessian quotient equation. Bao--Chen--Guan--Ji \cite{BCGJ} proved the Bernstein-type theorem, which states that any convex solution of such Hessian quotient equation on $\mathbb R^2$ is a quadratic polynomial. For further discussion on Hessian equation and Hessian quotient equation, we refer to \cite{Bao-Adv.China,Chang-Yu-Sigma2,2-Hessian,Du-2021necessary,Li-Ren-Wang-JFA,Nonpolynomial,Yu.Yuan1}
and the references therein.

For general $\tau\in[0,\frac{\pi}{2}]$,
Warren \cite{Warren} proved the Bernstein-type results under suitable semi-convex conditions. For higher dimension $n\geq 3$ case, exterior Bernstein-type result and optimal symmetric of solutions on $\mathbb R^n\setminus\{0\}$ are obtained earlier
in \cite{bao-liu2020asymptotic} and \cite{bao-liu2021symmetry} respectively
 by the authors. Especially in \cite{bao-liu2020asymptotic} we also obtained higher order expansions at infinity, which is parallel to the asymptotic expansions near isolated singular point of the Yamabe equation and $\sigma_k$-Yamabe equation obtained by  Han--Li--Li \cite{Han2019-Expansion} etc.

Our first main result considers the asymptotic behavior and expansion at infinity.
\begin{theorem}\label{thm:classify-constant}
  Let $u$ be a classical solution of \eqref{equ:Perturbed-General} in $\mathbb R^2\setminus\overline{\Omega}$ with condition \eqref{equ:condition-semiConvex}, where $\Omega\subset\mathbb R^2$ is a bounded domain.
  Then there exist $\gamma, d,d_1,d_2\in\mathbb R, \beta\in\mathbb R^2$ and $  A\in\mathtt{Sym}(2)$ with $F_{\tau}(\lambda(A))=C_0$ such that
  \begin{equation}\label{equ:result-Rough-cons}
  \begin{array}{llll}
  &\displaystyle u(x)-\left(\frac{1}{2}x^TAx+\beta x+\gamma\right)- d\ln (x^TQx)\\
  =&\displaystyle (x^TQx)^{-\frac{1}{2}}(d_1\cos\theta+d_2\sin\theta)+O_k(|x|^{-2}(\ln|x|))\\
  \end{array}
  \end{equation}
  as $|x|\rightarrow\infty$ for all $k\in\mathbb N$, where $
  \theta=\frac{Q^{\frac{1}{2}}x}{(x^TQx)^{\frac{1}{2}}}
  $  and
  \begin{equation}\label{equ:def-matrixQ}
  Q=(DF_{\tau}(\lambda(A)))^{-1}=\frac{1}{2}\left(\sin\tau A^2+2\cos\tau A+\sin\tau I\right).
  \end{equation}

\end{theorem}

\begin{remark}\label{Remark:value-of-d}
  For any $\Omega'\supset \Omega$, the constant $d$ is independent on the choice of $\Omega'$ and can be represented by the following, where $\nu$ denotes the outer normal vector on boundary of $\Omega'$ and $u_{i}, u_{ij}$ denote the partial differentials of $u$.
\begin{enumerate}[(i)]
  \item for $\tau=0$,
  \begin{equation}\label{equ:represent-d-MA}
d=\dfrac{1}{4\pi e^{C_0}}\left(
\int_{\partial\Omega'} u_{1}\left(u_{22},-u_{12}\right) \cdot \nu d s
-e^{2C_0}|\Omega'|
\right);
\end{equation}
\item for $\tau\in(0,\frac{\pi}{4})$,
  \begin{equation}\label{equ:represent-d-small}
  \begin{array}{llll}
d&=& \displaystyle \dfrac{1}{8\pi  \cdot \exp(\frac{b}{\sqrt{a^2+1}}C_0)}\int_{\partial\Omega'}(u_1+(a-b)x_1)(u_{22}+a-b,-u_{12})\cdot \nu ds\\
&&\displaystyle -\dfrac{\exp(\frac{b}{\sqrt{a^2+1}}C_0)}{8\pi }\int_{\partial\Omega'}(u_1+(a+b)x_1)(u_{22}+a+b,-u_{12})\cdot \nu ds;\\
\end{array}
\end{equation}
\item for $\tau=\frac{\pi}{4}$,
\begin{equation}\label{equ:represent-d-mid}
\begin{array}{llll}
  d&=&\displaystyle \dfrac{1}{2\pi}|\Omega'|-
  \dfrac{\sqrt 2 C_0}{8\pi}\int_{\partial\Omega'}(u_1+1)(u_{22}+1,-u_{12})\cdot\nu ds\\
  & &\displaystyle +\dfrac{1}{4\pi}\int_{\partial\Omega'} (u_1+1,u_2+1)\cdot \nu ds;
\end{array}
\end{equation}
\item for $\tau\in(\frac{\pi}{4},\frac{\pi}{2})$,
\begin{equation}\label{equ:represent-d-large}
\begin{array}{lll}
d&=&\displaystyle  \frac{b}{4 \pi}\cos (\frac{b}{\sqrt{a^2+1}}C_0) \int_{\partial \Omega'}  (u_1+ax_1,u_2+ax_2)\cdot \nu ds\\
&&+\displaystyle \dfrac{b}{4\pi}\sin (\frac{b}{\sqrt{a^2+1}}C_0) \int_{\partial\Omega'}( u_{1}+a_1x_1)\left(u_{22}+a,-u_{12}\right) \cdot \nu d s\\
&&\displaystyle -\dfrac{b}{4\pi}\sin (\frac{b}{\sqrt{a^2+1}}C_0) |\Omega'|;
\end{array}
\end{equation}
\item for $\tau=\frac{\pi}{2}$,
\begin{equation}\label{equ:represent-d-SPL}
d=\frac{1}{4 \pi}\left(\int_{\partial \Omega'} \left(\cos C_0 (u_1,u_2)\cdot \nu+\sin C_0 u_{1}\left(u_{22},-u_{12}\right) \cdot \nu\right) d s-\sin C_0 |\Omega'|\right).
\end{equation}
\end{enumerate}
\end{remark}

\begin{remark}
  For a given solution $u\in C^2(\mathbb R^2\setminus\overline\Omega)$, let
  $$
  w(x):=u- \frac{1}{2} x^{T} A x-\beta x-\gamma-d\ln (x^TQx)\quad\text{and}\quad
  W(x):=w(Q^{-\frac{1}{2}}x)
  $$
  where $A,\beta,\gamma,d,Q$ are as in Theorem \ref{thm:classify-constant} and Remark \ref{Remark:value-of-d}. Then $d_1,d_2$ can be represented as
  $$
  d_1=\lim_{r\rightarrow+\infty}\int_{\partial B_r}W(x)\cos \frac{x}{|x|}ds,\quad
  d_2=\lim_{r\rightarrow+\infty}\int_{\partial B_r}W(x)\sin \frac{x}{|x|}ds
  $$
  respectively.
\end{remark}

Our second main result considers the symmetry of solution on punctured space.
\begin{theorem}\label{Thm:classify-n=2}
  Let $u $ be a classical solution of \eqref{equ:Perturbed-General} in $\mathbb R^2\setminus\{0\}$
  with condition \eqref{equ:condition-semiConvex}.
  Then $u\in C^0(\mathbb R^2)$,
  $$
  F_{\tau}\left(\lambda\left(D^{2} u\right)\right)\geq C_0\quad\text{in }\mathbb R^2
  $$
  in viscosity sense and for the $A\in\mathtt{Sym}(2)$, $\beta\in\mathbb R^2$ from Theorem \ref{thm:classify-constant} such that $u(x)-\beta x$ is symmetric in eigenvector directions of $A$, i.e.,
  \begin{equation}\label{equ:symmetry}
  u(\widetilde x)-\beta \widetilde x=u(x)-\beta x,\quad\forall~\{\widetilde x\in\mathbb R^2:(O^T\widetilde x)_i=\pm(O^T x)_i,\quad\forall~i=1,2\},
  \end{equation}
  where $O$ is an orthogonal matrix such that $OAO^T$ is diagonal.
\end{theorem}
\begin{remark}
  As in Proposition 1.5 of our earlier work \cite{bao-liu2021symmetry}, the symmetry in Theorem \ref{Thm:classify-n=2} is optimal for $\tau\in(0,\frac{\pi}{2}]$ in the sense  that if $u$ satisfies
  $$
  u(x)+\frac{1}{2}K|x|^2=U\left(\frac{1}{2}x^T(A+KI)x\right),\quad K=\left\{
  \begin{array}{llll}
    a-b, & 0<\tau<\frac{\pi}{4},\\
    1, & \tau=\frac{\pi}{4},\\
    a, & \frac{\pi}{4}<\tau<\frac{\pi}{2},\\
    0, & \tau=\frac{\pi}{2},\\
  \end{array}
  \right.
  $$
  for some scalar function $U:\mathbb R\rightarrow\mathbb R$ and $A+KI>0$, then $u$ is either a quadratic polynomial or radially symmetric. For $\tau=0$ case, as pointed out in  J\"orgens \cite{Jorgens} and Jin--Xiong \cite{Jin-Xiong}, any convex solution of \eqref{equ:Perturbed-General} in $\mathbb R^2\setminus\{0\}$ with $\tau=0$ can be represented by
  $$
  u=e^{C_0}\int_0^{(x^TAx)^{\frac{1}{2}}}(r^2+c_1)^{\frac{1}{2}}dr+c_2
  $$
  for some $c_1\geq 0, c_2\in\mathbb R$ and $0<A\in\mathtt{Sym}(2)$ with $\det A=e^{2C_0}$.
\end{remark}

\begin{remark}
 For $C_0=0$ and $\tau=\frac{\pi}{2}$, Eq.~\eqref{equ:Perturbed-General} can be translated into the harmonic equation $\Delta u=0$. Similar asymptotic and symmetry results hold under various extra conditions can be found in \cite{Leifu-Bernstein,20} etc.
\end{remark}
\begin{remark}\label{rem:final}
  In fact, the arguments in Theorem \ref{thm:classify-constant} hold for a more general class of equation. Let $u$ be a classical solution of
  \begin{equation}\label{equ:general-threeterm}
    c_2\lambda_1\lambda_2+c_1(\lambda_1+\lambda_2)+c_0=0\quad\text{in }\mathbb R^2\setminus\overline\Omega,
  \end{equation}
  where $c_0,c_1,c_2\in\mathbb R$ satisfies
  \begin{equation}\label{cond:structure}
  c_2\not=0\quad\text{and}\quad c_0c_2<c_1^2.
  \end{equation}
  By a direct computation, all five situations in Theorem \ref{thm:classify-constant} are concluded in \eqref{equ:general-threeterm}.
\end{remark}

The paper is organized as follows. In sections \ref{sec-asymptotic-2} and \ref{sec-asymptotic-refine} we prove the asymptotic expansion result in Theorem \ref{thm:classify-constant}. In section \ref{sec-symmetric} we prove Theorem \ref{Thm:classify-n=2}. In section \ref{sec-remark} we prove Remark \ref{Remark:value-of-d}.

\section{Asymptotic behavior at infinity}\label{sec-asymptotic-2}

In this section, we prove that under the conditions of Theorems \ref{thm:classify-constant}, there exist $\gamma, d\in\mathbb R, \beta\in\mathbb R^2, A\in\mathtt{Sym}(2)$ with $F_{\tau}(\lambda(A))=C_0$ such that
\begin{equation}\label{equ:asym-Behavior}
u(x)= \frac{1}{2}x^TAx+\beta x+\gamma +d\ln (x^TQx)+O_k(|x|^{-1})
\end{equation}
as $|x| \rightarrow \infty$ for all $k \in \mathbb{N}$, where $Q$ is as in \eqref{equ:def-matrixQ}.

By interior regularity as Lemma 17.16 of \cite{GT}, we have $u\in C^{k}(\mathbb R^2\setminus\overline{\Omega})$ for any $k\geq 2$. Suppose  $\Omega\subset B_{R}$, where $B_R$ denote the ball with radius $R$ and centered the origin.
Since $u$ satisfies semi-convex condition \eqref{equ:condition-semiConvex}, then
$$
v:=\left\{
\begin{array}{lllll}
  u, & \tau=0,\\
  u+\frac{a-b}{2}|x|^2, & \tau\in(0,\frac{\pi}{4}),\\
  u+\frac{1}{2}|x|^2, & \tau=\frac{\pi}{4},\\
  u+\frac{a+b}{2}|x|^2, & \tau\in(\frac{\pi}{4},\frac{\pi}{2}),\\
\end{array}
\right.
$$
has positive definite Hessian matrix on an open neighbourhood of $\overline{B_{R+2}\setminus B_{R+1}}$. By Theorem 3.2 of  \cite{Min-ExtensionConvexFunc}, $v|_{\overline{B_{R+2}\setminus B_{R+1}}}$ can be extended to a $C^k$ function on the convex hull $\overline{B_{R+2}}$ with positive definite Hessian. Consequently, this provides an extension of $u|_{\overline{B_{R+2}\setminus B_{R+1}}}$ to $\overline{B_{R+2}}$ that preserves the semi-convexity. Thus hereinafter, we may assume without loss of generality that $u\in C^k(\mathbb R^2)$ for any $k\geq 2$ and condition \eqref{equ:condition-semiConvex} holds on entire $\mathbb R^2$.

 For $\tau=0$ case, we have
 $Q=A$ and
 \eqref{equ:asym-Behavior} is
proved in Theorem 1.2 of \cite{CL} and Theorem 1.1 of \cite{Bao-Li-ZhangMA}.   For $\tau=\frac{\pi}{4}$ and $\tau=\frac{\pi}{2}$ cases, \eqref{equ:asym-Behavior} follows from Theorems 4.3 and 1.1  of \cite{Li-Li-Yuan-Bernstein-SPL} respectively.

\begin{proof}[Proof of \eqref{equ:asym-Behavior}, $\tau\in(0,\frac{\pi}{4})$ case]
  Let $$
\overline{u}(x):=u(x)+\dfrac{a+b}{2}|x|^2,
$$
then
\begin{equation*}
D^{2} \overline{u}=D^{2} u+(a+b) I>2bI\quad\text{in }\mathbb{R}^2.
\end{equation*}
Let $(\widetilde{x},\widetilde{v})$ be the Legendre transform of $(x,\overline{u})$, i.e.,
$$
\left\{
\begin{array}{ccc}
  \widetilde{x}:=D\overline{u}(x)=Du(x)+(a+b)x,\\
  D v(\widetilde{x}):=x\\
\end{array}
\right.
$$
and  we have
\begin{equation*}
D^{2} v(\widetilde{x})=\left(D^{2} \overline{u}(x)\right)^{-1}=(D^2u(x)+(a+b)I)^{-1}<\frac{1}{2b}I.
\end{equation*}
Let \begin{equation}\label{LegendreTransform}
\widetilde{u}(\widetilde{x}):=\dfrac{1}{2}|\widetilde{x}|^2-2bv(\widetilde{x}).
\end{equation}
By a direct computation,
\begin{equation}\label{property-Legendre}
\widetilde{\lambda}_{i}\left(D^{2} \widetilde{u}\right)=1-2 b \cdot \frac{1}{\lambda_{i}+a+b}=\frac{\lambda_{i}+a-b}{\lambda_{i}+a+b}\in (0,1).
\end{equation}
Thus $\widetilde{u}(\widetilde{x})$ satisfies the following Monge--Amp\`ere type equation
\begin{equation}\label{temp-16}
\sum _ { i = 1 } ^ { 2 } \ln \widetilde{\lambda_i}=\frac{2b}{\sqrt{a^2+1}}C_0,\quad \text{in }\mathbb{R}^2\setminus\overline{\widetilde{\Omega}}
\end{equation}
for some bounded set $\widetilde{\Omega}=D\overline{u}(\Omega')$.

By Theorem 1.2 of \cite{CL}, we have
\begin{equation}\label{equ:temp-1}
\limsup_{|\widetilde x|\rightarrow\infty}|\widetilde x|^{k+1}\left|
D^k\left(
\widetilde u(\widetilde x)-\left(
\frac{1}{2}\widetilde x^T\widetilde A\widetilde x+\widetilde \beta\widetilde x+\widetilde d\ln (\widetilde x^T\widetilde A\widetilde x)+\widetilde \gamma
\right)
\right)
\right|<\infty
\end{equation}
for all $k\in\mathbb N$
for some $\widetilde\beta\in\mathbb R^n,\widetilde \gamma,\widetilde d\in\mathbb R$ and $\widetilde A\in\mathtt{Sym}(2)$ satisfying $\det \widetilde A=\frac{2b}{\sqrt{a^2+1}}C_0$ and $0<\widetilde A\leq I$. Now we prove that all eigenvalues of $\widetilde A$ are strictly less than $1$, which implies $I-\widetilde A$ is an invertible matrix. By contradiction and rotating the $\widetilde{x}$-space to make $\widetilde{A}$ diagonal, we suppose that that $\widetilde{A}_{11}=$
  $1.$ Then by \eqref{LegendreTransform} and the asymptotic behavior of $D\widetilde{u}$, there exists $\widetilde{\beta}_1\in\mathbb{R}$ such that
\begin{equation}\label{strip-argument}
x_1=D_{1}v(\widetilde{x})=\widetilde{\beta}_{1}+O\left(|\widetilde{x}|^{-1}\right)
\end{equation}
as $|\widetilde{x}|\rightarrow\infty$. By triangle inequality and semi-convex condition in \eqref{equ:condition-semiConvex},
\begin{equation}\label{limitofX}
|\widetilde{x}|\geq
-|\widetilde{0}|+|\widetilde{x}-\widetilde{0}|
> -|\widetilde{0}|+2b|x|,
\end{equation}
\eqref{strip-argument} implies that $\mathbb{R}^{2}$ is bounded in the $x_{1}$-direction, hence  a contradiction.  Thus $\lambda_i(\widetilde{A})<1$ for $i=1,2$. Such ``strip argument'' is also used in  \cite{Li-Li-Yuan-Bernstein-SPL} etc.

Consequently,
$$
D^2u(x)=2b(I-\widetilde{A})^{-1}-(a+b)I+O(|x|^{-2})
$$
as $|x|\rightarrow\infty$.
Substituting the asymptotic behavior \eqref{equ:temp-1} with $k=1$ into the inverse Legendre transform, we have
$$
x=Dv(\widetilde x)=\frac{1}{2b}\left(
I-\widetilde A+\frac{\widetilde d\widetilde A}{\widetilde x^T\widetilde A\widetilde x}
\right)\widetilde x-\frac{\widetilde \beta}{2b}+O(|\widetilde x|^{-2}).
$$
Since the eigenvalues of $\widetilde A$ is bounded away from $1$, we have
$$
\widetilde x=\left(
I-\widetilde A
\right)^{-1}\left(2b x+\widetilde \beta-\frac{\widetilde d\widetilde A\widetilde x}{\widetilde x^T\widetilde A\widetilde x}\right)+O(|\widetilde x|^{-2}).
$$
Iterate and using $0<\widetilde A<I$ we have
$$
\frac{\widetilde d\widetilde A\widetilde x}{\widetilde x^T\widetilde A\widetilde x}
=\dfrac{\widetilde d\widetilde A(I-\widetilde A)^{-1}(2bx+\widetilde\beta)}{
(2bx+\widetilde\beta)^T(I-\widetilde A)^{-1}\widetilde A(I-\widetilde A)^{-1}(2bx+\widetilde\beta)
}+O(|\widetilde x|^{-2})
$$
as $|\widetilde x|\rightarrow\infty.$ Together with \eqref{limitofX},
$$
\begin{array}{llll}
&D\bar u(x)=\widetilde x\\
=&\displaystyle (I-\widetilde A)^{-1}(2bx+\widetilde\beta)
-\widetilde d\dfrac{(I-\widetilde A)^{-1}\widetilde A(I-\widetilde A)^{-1}(2bx+\widetilde\beta)}{
(2bx+\widetilde\beta)^T(I-\widetilde A)^{-1}\widetilde A(I-\widetilde A)^{-1}(2bx+\widetilde\beta)
}+O(|x|^{-2})\\
=&\displaystyle
2b(I-\widetilde A)^{-1}x+(I-\widetilde A)^{-1}\widetilde\beta
-\widetilde d\dfrac{(I-\widetilde A)^{-1}\widetilde A(I-\widetilde A)^{-1}(2bx+\widetilde\beta)}{
(2bx+\widetilde\beta)^T(I-\widetilde A)^{-1}\widetilde A(I-\widetilde A)^{-1}(2bx+\widetilde\beta)
}+O(|x|^{-2}).\\
\end{array}
$$
Integrating term by term and letting $A:=2b(I-\widetilde A)^{-1}-(a+b)I$ we have $\beta\in\mathbb R^n, \gamma, d\in\mathbb R$ such that
$$
\begin{array}{llll}
& \bar u(x)\\
=& bx^T(I-\widetilde A)^{-1}x+\beta x+d\ln (2bx^T(I-\widetilde A)^{-1}\widetilde A(I-\widetilde A)^{-1}2bx)+\gamma+O(|x|^{-1})\\
=& \frac{1}{2}x^T(A+(a+b)I)x+\beta x+d\ln (x^T\overline Qx)+\gamma+O(|x|^{-1})
\end{array}
$$
as $|x|\rightarrow\infty$, where
$$
\begin{array}{llll}
\overline Q &= & (A+(a+b)I)\left(I-2b(A+(a+b)I)^{-1}\right)(A+(a+b)I)\\
&=& (A+(a+b)I)^2-2b(A+(a+b)I)\\
&=& A^2+2aA+(a^2-b^2)I.
\end{array}
$$
By the definition of $a,b$, we have
$$
\overline Q=\dfrac{1}{\sin \tau}\left(\sin\tau A^2+2\cos \tau A+\sin\tau I\right)=\dfrac{2}{\sin\tau}Q
$$
and hence
$$
d\ln (x^T\overline Qx)+\gamma=d\ln (x^TQx)+\gamma-d\ln\frac{\sin\tau}{2}.
$$
The result follows immediately.
\end{proof}
\begin{proof}[Proof of \eqref{equ:asym-Behavior}, $\tau\in(\frac{\pi}{4},\frac{\pi}{2})$ case]

By a direct computation (see for instance \cite{huang2019entire,Warren}), if
 $\lambda_i>-a-b,$ for $i=1,2$, then
  \begin{equation}\label{equ:identity}
  \sum_{i=1}^2\arctan \frac{\lambda_{i}+a-b}{\lambda_{i}+a+b}=\sum_{i=1}^2\arctan \left(\frac{\lambda_{i}+a}{b}\right)-\frac{\pi}{2}.
  \end{equation}
Let
  \begin{equation*}
v(x):=\frac{u(x)}{b}+\frac{a}{2 b}|x|^{2}.
\end{equation*}
 By a direct computation,
  \begin{equation*}
\sum_{i=1}^2\arctan \lambda_i(D^2v)=\frac{b}{\sqrt{a^2+1}}C_0
+\frac{\pi}{2}
\end{equation*}
and $D^2v>-I$
in $\mathbb R^2$. For $C_0\not=-\frac{\pi}{2}\frac{\sqrt{a^2+1}}{b}$, the result follows immediately by $\tau=\frac{\pi}{2}$ case of Theorem \ref{thm:classify-constant}. More rigorously from the result in $\tau=\frac{\pi}{2}$ case, we have
$$
v=\left(\frac{1}{2}x^T\overline Ax+\overline\beta x+\overline\gamma\right)+\overline d\ln (x^T(\overline A^2+I)x)+O_k(|x|^{-1})
$$
as $|x|\rightarrow\infty$ for some $\overline{\gamma},\overline d,\in\mathbb R, \overline\beta\in\mathbb R^2$ and $\overline A\in\mathtt{Sym}(2)$. Consequently,
$$
\begin{array}{llll}
  u(x) &=& \displaystyle \left(\frac{1}{2}x^T
  (b\overline A-aI)x+b\overline\beta x+b\overline\gamma\right)+b\overline d\ln (x^T(\overline A^2+I)x)+O_k(|x|^{-1})\\
  &=:& \displaystyle\left(\frac{1}{2}x^TAx+\beta x+\gamma\right)+d\ln (x^TQx)+O_k(|x|^{-1}),\\
\end{array}
$$
where $Q$ is as in \eqref{equ:def-matrixQ}, $$
A=b\overline A-aI,\quad\beta=b\overline\beta,\quad
\overline A^2+I=\dfrac{2}{b^2\sin\tau}Q,\quad
\gamma=b\overline\gamma-b\overline d\ln\frac{b^2\sin\tau}{2}
\quad\text{and}\quad d=b\overline d.
$$
The result follows immediately.

For $C_0=-\frac{\pi}{2}\frac{\sqrt{a^2+1}}{b}$, then $v$ is harmonic on exterior domain of $\mathbb R^2$ and $D^2v>-I$. In this case, the result follows by Bernstein-type results as in \cite{Leifu-Bernstein,20} etc. More rigorously, by differentiating the equation twice and Theorem 3 from \cite{20}, there exists $A\in\mathtt{Sym}(2)$ with $\lambda_1(A)+\lambda_2(A)=0$ such that $D^2v\rightarrow A$ as $|x|\rightarrow\infty$. By a direct computation, $v-\frac{1}{2}x^TAx=o(|x|^2)$ and is harmonic on exterior domain of $\mathbb R^2$. Hence by spherical harmonic expansions as in (2.31) and (2.32) of \cite{Bao-Li-ZhangMA}, $v$ converge to a quadratic function with additional $\ln$ term. This finishes the proof of desired result.
\end{proof}

We finish this section by proving an analogue for Eq.~\eqref{equ:general-threeterm}, which will be used to proved Remark \ref{rem:final}. We may assume without loss of generality that $c_2>0$, otherwise we rewrite Eq.~\eqref{equ:general-threeterm} into
$$
-c_2\lambda_1\lambda_2-c_1(\lambda_1+\lambda_2)-c_0=0,
$$
which still satisfies \eqref{cond:structure}. By a direct computation and condition \eqref{cond:structure}, Eq.~\eqref{equ:general-threeterm} can be written into
$$
\left(\lambda_1+\frac{c_1}{c_2}\right)\cdot \left(\lambda_2+\frac{c_1}{c_2}\right)=\dfrac{c_1^2-c_0c_2}{c_2^2}>0\quad\text{in }\mathbb R^2\setminus\overline\Omega.
$$
Thus we may assume without loss of generality that
$$
D^2u>-\frac{c_1}{c_2}I,
$$
otherwise we consider $-u$ instead.
Let
$$
\bar u(x):=u(x)+\dfrac{c_1}{c_2}|x|^2
$$
and $(\widetilde x,\widetilde u)$ be the Legendre transform of $(x,\bar u)$. By a direct computation, $\widetilde u$ is convex and satisfies the following Monge--Amp\`ere type equation
$$
\det (D^2\widetilde u)=\dfrac{c_2^2}{c_1^2-c_0c_2}.
$$
Then the desired result follows by Theorem 1.2 of \cite{Caffarelli-Li-2003} and strip argument as above.

\section{Refined asymptotics}\label{sec-asymptotic-refine}

In this section, we prove Theorem \ref{thm:classify-constant} by the asymptotic behavior proved in \eqref{equ:asym-Behavior} and analysis on linear elliptic equations. The linear elliptic equation below comes from the linearized equation satisfied by
$$
v(x)=u(x)-\left(\frac{1}{2} x^{T} A x+\beta x+\gamma\right).
$$
By \eqref{equ:asym-Behavior}, we may assume (when needed) that
\begin{equation}\label{equ:asym-decomp}
v=d\ln (x^TQx)+O_l(|x|^{-1})
\end{equation}
as $|x|\rightarrow+\infty$
for some positive matrix $Q$ and $d\in\mathbb R$.

\begin{lemma}\label{Lem-fastConverge}
  Let $g\in C^{\infty}(\mathbb R^2)$ satisfy
\begin{equation}\label{equ:vanishing-g-Lp}
\|g(r \cdot)\|_{L^{p}\left(\mathbb{S}^{1}\right)} \leq c_{0} r^{-k_{1}}(\ln r)^{k_{2}} \quad \forall ~r>1
\end{equation}
for some $c_0>0, k_1>2, k_2\geq 0$ and $p\geq 2$. Then there exists a smooth solution $v$ of
\begin{equation}\label{equ:Laplace}
  \Delta v=g\quad\text{in }\mathbb R^2\setminus\overline{B_1}
\end{equation}
such that
\begin{equation}\label{equ:FastVanishing}
|v(x)|  \leq Cc_0|x|^{2-k_1}(\ln|x|)^{k_2+1}
\end{equation}
for $|x|>1$ for some $C>0$.
\end{lemma}
\begin{proof}
In polar coordinate we have
$$
\Delta v=\dfrac{\partial^2v}{\partial r^2}+\frac{1}{r}\frac{\partial v}{\partial r}+\frac{1}{r^{2}} \frac{\partial^{2} v}{\partial \theta^{2}},
$$
where $r$ represents the radial distance and $\theta$ the angle.
Let
$$
Y_1^{(0)}(\theta)\equiv \dfrac{1}{\sqrt{2\pi}},\quad Y_1^{(k)}(\theta)=\dfrac{1}{\sqrt \pi}\cos k\theta\quad\text{and}\quad
Y_2^{(k)}(\theta)=\dfrac{1}{\sqrt \pi}\sin k\theta,
$$
which forms a complete  standard orthogonal basis of $L^2(\mathbb S^1)$.
Decompose $g$ and the wanted solution $v$
into
\begin{equation}\label{equ-star}
v(x)=a_{0,1}(r)+\sum_{k=1}^{+\infty}\sum_{m=1}^{2} a_{k, m}(r) Y_{m}^{(k)}(\theta),\quad
g(x)=b_{0,1}(r)+\sum_{k=1}^{+\infty}\sum_{m=1}^{2} b_{k, m}(r) Y_{m}^{(k)}(\theta),
\end{equation}
where
 $r=|x|, \theta=\frac{x}{|x|}$ and $$a_{k,m}(r):=\int_{\mathbb{S}^{n-1}} v(r \theta) \cdot Y_{m}^{(k)}(\theta) \mathtt{d} \theta,\quad b_{k,m}(r):=\int_{\mathbb{S}^{n-1}} g(r\theta) \cdot Y_{m}^{(k)}(\theta) \mathtt{d} \theta.$$
By the linear independence of $Y_m^{(k)}(\theta)$, \eqref{equ:Laplace} implies that
$$
a_{0,1}''(r)+\frac{1}{r}a_{0,1}'(r)=b_{0,1}(r)\quad\text{in }r>1
$$
and for all $k\in\mathbb{N}_*$ with $m=1,2,$
\begin{equation}\label{Equ-equ-equ}
a_{k,m}^{\prime \prime}(r)+\frac{1}{r} a_{ k,m}^{\prime}(r)-\frac{k^2}{r^{2}} a_{k,m}(r) =b_{k,m}(r)\quad\text{in }r>1.
\end{equation}

By solving the ODE, there exist constants $C_{k,m}^{(1)},C_{k,m}^{(2)}$ such that for all $r>1$,
\begin{equation}\label{Equ-def-Wronski}
\begin{array}{lll}
  a_{k,m}(r)&=&C_{k,m}^{(1)}r^k+
C_{k,m}^{(2)}r^{-k}\\
&&\displaystyle
-\dfrac{1}{2k}r^k\int_{2}^r\tau^{1-k}b_{k,m}(\tau)\mathtt{d} \tau
+\dfrac{1}{2k}r^{-k}\int_{2}^r\tau^{1+k}b_{k,m}(\tau)\mathtt{d} \tau
\end{array}
\end{equation}
for $k\geq 1$
and
$$
\begin{array}{lllll}
a_{0,1}(r)&=& C_{0,1}^{(1)}+C_{0,1}^{(2)}\ln r\\
&&\displaystyle -\int_2^r \tau\ln\tau  b_{0,1}(\tau)d \tau+\ln r\int_2^r \tau b_{0,1}(\tau) d \tau.
\end{array}
$$
 By \eqref{equ:vanishing-g-Lp},
\begin{equation}\label{Equ-converge}
|b_{0,1}(r)|^2+\sum_{k=1}^{+\infty}\sum_{m=1}^{2}|b_{k,m}(r)|^2=||g(r\cdot)||^2_{L^2(\mathbb{S}^{n-1})}
\leq c_0^2(2\pi)^{\frac{p-2}{p}}r^{-2k_1}(\ln r)^{2k_2}
\end{equation}
for all $r>1$. Then $r^{1-k}b_{k,m}(r)\in L^1(2,+\infty)$  for all $k\geq 1$ and $r^{k+1}b_{k,m}(r)\in L^1(2,+\infty)$  for all
$1\leq k<k_1-2$. We choose $C_{k,m}^{(1)}$ and $C_{k,m}^{(2)}$ in \eqref{Equ-def-Wronski} such that
\begin{equation}\label{equ:def-v-4}
  a_{0,1}(r):=-\int_{+\infty}^r\tau\ln \tau b_{0,1}(\tau)d\tau+\ln r\int_{+\infty}^r\tau b_{0,1}(\tau)d\tau,
\end{equation}
\begin{equation}\label{Equ-def-v-2}
a_{k,m} (r):=
- \dfrac{1}{2k}r^k\int_{+\infty}^r\tau^{1-k}b_{k,m}(\tau)\mathtt{d} \tau
+ \dfrac{1}{2k}r^{-k}\int_{+\infty}^r\tau^{1+k}b_{k,m}(\tau)\mathtt{d} \tau
\end{equation}
for all $1\leq k<k_1-2$ and
\begin{equation}\label{Equ-def-v-3}
a_{k,m} (r):=
- \dfrac{1}{2k}r^k\int_{+\infty}^r\tau^{1-k}b_{k,m}(\tau)\mathtt{d} \tau
+ \dfrac{1}{2k}r^{-k}\int_{2}^r\tau^{1+k}b_{k,m}(\tau)\mathtt{d} \tau
\end{equation}
for all $k\geq \max\{1,k_1-2\}$.

To prove that the series $v(x)$ defined by \eqref{equ-star} converges and obtain its convergence speed, we separate into three cases according to the value of $k_1-2$. Hereinafter, we let $[k]$ denote the largest natural number no larger than $k$, especially for $k\leq 0$, $[k]=0$.

For $k_1-2\not\in\mathbb N$, we pick $0<\varepsilon:=\frac{1}{2}\min\{1,\mathtt{dist}(k_1-2,\mathbb N)\}$ such that
$$
\left\{
\begin{array}{lll}
3-2k_1+\varepsilon<-1,\\
3+2k-2k_1+\varepsilon<-1 & \forall~ 1\leq k\leq [k_1-1]-1,\\ 3+2k-2k_1-\varepsilon>-1 & \forall~ k\geq [k_1-1].
\end{array}
\right.
$$
Thus \eqref{Equ-converge} implies$$
\begin{array}{llll}
&\displaystyle a_{0,1}^2(r)+\sum_{k=1}^{+\infty}\sum_{m=1}^{2}a_{k,m}^2(r)\\
  \leq & \displaystyle  2\left|\int_{+\infty}^{r} \tau \ln \tau b_{0,1}(\tau) d \tau\right|^2+2\left|\ln r\int_{+\infty}^r\tau b_{0,1}(\tau)d\tau\right|^2
  +2\sum_{k=1}^{+\infty}\sum_{m=1}^2r^{2k}\left|
  \int_{+\infty}^r\tau^{1-k}b_{k,m}(\tau)d\tau
  \right|^2\\
  &\displaystyle
  +2\sum_{k=1}^{[k_1-1]-1}\sum_{m=1}^2r^{-2k}\left|\int_{+\infty}^r\tau^{1+k}b_{k,m}(\tau)d\tau\right|^2
  +2\sum_{k=[k_1-1]}^{+\infty}\sum_{m=1}^2r^{-2k}\left|
\int_2^r\tau^{1+k}b_{k,m}(\tau)d\tau
\right|^2\\
\leq & \displaystyle 2\int_r^{+\infty}\tau^{-1-\epsilon}(\ln\tau)^2 d\tau \int_r^{+\infty}\tau^{3+\epsilon}b_{0,1}^2(\tau)d\tau+2(\ln r)^2\int_r^{+\infty} \tau^{-1-\epsilon}d\tau \int_r^{+\infty}\tau^{3+\epsilon}b_{0,1}^2(\tau)d\tau\\
&+2\displaystyle\sum_{k=1}^{+\infty}\sum_{m=1}^2r^{2k}\int_r^{+\infty}\tau^{-1-2k-\epsilon}d\tau
\cdot\int_r^{+\infty}\tau^{3+\epsilon}b_{k,m}^2(\tau)d\tau\\
&+2\displaystyle \sum_{k=1}^{[k_1-1]-1}\sum_{m=1}^2r^{-2k}\int_r^{+\infty}\tau^{3+2k-2k_1+\epsilon}(\ln\tau)^{2k_2}d\tau\cdot
\int_r^{+\infty}\tau^{2k_1}(\ln\tau)^{-2k_2}b_{k,m}^2(\tau)\frac{d\tau}{\tau^{1+\epsilon}}\\
\end{array}
$$
$$
\begin{array}{llll}
&+2\displaystyle
\sum_{k=[k_1-1]}^{+\infty}\sum_{m=1}^2r^{-2k}\int_2^r\tau^{3+2k-2k_1-\epsilon}(\ln\tau)^{2k_2}d\tau
\cdot \int_2^r\tau^{2k_1}(\ln\tau)^{-2k_2}b_{k,m}^2(\tau)\frac{d\tau}{\tau^{1-\epsilon}}\\
\leq & \displaystyle Cr^{-\epsilon}(\ln r)^2 \int_r^{+\infty}\tau^{3+\epsilon}\left(b_{0,1}^2(\tau)
+\sum_{k=1}^{\infty}\sum_{m=1}^{2}b_{k,m}^2(\tau)\right)d\tau \\
&\displaystyle +Cr^{4-2k_1+\epsilon}(\ln\tau)^{2k_2}\cdot \int_r^{+\infty}\tau^{2k_1}(\ln\tau)^{-2k_2}
\sum_{k=1}^{[k_1-1]-1}\sum_{m=1}^2
b_{k,m}^2(\tau)\frac{d\tau}{\tau^{1+\epsilon}}\\
&\displaystyle +Cr^{4-2k_1-\epsilon}(\ln\tau)^{2k_2}\cdot \int_2^r\tau^{2k_1}(\ln\tau)^{-2k_2}
\sum_{k=[k_1-1]}^{+\infty}\sum_{m=1}^2
b_{k,m}^2(\tau)\frac{d\tau}{\tau^{1-\epsilon}}\\
\leq & Cc_0^2r^{4-2k_1}(\ln r)^{2k_2+2}.\\
\end{array}
$$
For $k_1-2\in\mathbb N$, we pick $\varepsilon:=\frac{1}{2}$. Then
$$
\left\{
\begin{array}{lll}
3-2k_1+\varepsilon<-1,\\
3+2k-2k_1+\varepsilon<-1 & \forall~ 1\leq k\leq k_1-3,\\
3+2k-2k_1-\varepsilon>-1 & \forall~ k\geq k_1-1.
\end{array}
\right.
$$
Similar to the calculus above, \eqref{Equ-converge} implies
\begin{equation}\label{equ-estimate-4}
\begin{array}{llll}
  &\displaystyle a_{0,1}^2(r)+\sum_{k=1}^{+\infty} \sum_{m=1}^{2} a_{k, m}^{2}(r)\\
  \leq & \displaystyle Cc_0^2  r^{4-2k_1}(\ln r)^{2k_2+2}+C
   \sum_{m=1}^{2}r^{4-2k_1}\left|\int_2^r\tau^{k_1-1}b_{k_1-2,m}(\tau) d\tau\right|^2\\
  \leq &  Cc_0^2  r^{4-2k_1}(\ln r)^{2k_2+2}\\
  & +\displaystyle C\sum_{m=1}^{2}
  \int_2^r\tau^{2k_1}(\ln\tau)^{-2k_2}b_{k_1-2,m}^2(\tau)\dfrac{ d\tau}{\tau}
  \int_2^r\tau^{-1}(\ln\tau)^{2k_2} d\tau\\
  \leq & Cc_0^2 r^{4-2k_1}(\ln r)^{2k_2+2}.\\
\end{array}
\end{equation}
This proves that $v(r)$ is well-defined, is a solution of \eqref{equ:Laplace} in distribution sense \cite{Gunther-ConfoNormCoord} and satisfies
\begin{equation}\label{equ-estimate-spherical}
||v(r\cdot)||^2_{L^2(\mathbb S^{1})}\leq Cc_0^2 r^{4-2k_1}(\ln r)^{2k_2+2}
\end{equation}
By interior regularity theory of elliptic differential equations,  $v$ is smooth \cite{GT}. It remains to prove the pointwise decay rate at infinity.

For any $r\gg 1$, we set
$$
v_r(x):=v(rx)\quad\forall~x\in B_4\setminus B_{1}=:D.
$$
Then $v_r$ satisfies
\begin{equation}\label{Equ-scaled}
\Delta v_r=r^2g(rx)=:g_r(x)\quad\text{in}~D.
\end{equation}
By weak Harnack inequality (see for instance Theorem 8.17 of \cite{GT}, see also (2.11) of \cite{Gunther-ConfoNormCoord}),
$$
\sup_{2<|x|<3}|v_r(x)|\leq C(p)\cdot \left(
||v_r||_{L^2(D)}+||g_r||_{L^p(D)}
\right).
$$
By \eqref{equ-estimate-spherical},
$$
\begin{array}{llll}
  ||v_r||_{L^2(D)}^2 & = & \displaystyle  \dfrac{1}{r^2}\int_{B_{4r}\setminus B_{r}}|v(x)|^2 d x\\
  &=&\displaystyle r^{-2}\int_{ r}^{4r}||v(\tau\theta)||_{L^2(\mathbb S^{1})}^2\cdot \tau d \tau\\
  &\leq & Cc_0^2  r^{4-2k_1}(\ln r)^{2k_2+2}.\\
\end{array}
$$
By \eqref{equ:vanishing-g-Lp},
$$
\begin{array}{llll}
  ||g_r||_{L^p(D)}^p & = &\displaystyle
  \dfrac{r^{2p}}{r^2}\int_{B_{4r}\setminus B_{ r}}|g(x)|^p d x\\
  &\leq & \displaystyle Cc_0^p   r^{2p-2}\int_{ r}^{4r}\tau^{-pk_1}(\ln\tau)^{pk_2}\cdot \tau d \tau\\
  &\leq & Cc_0^p  r^{2p-pk_1}(\ln r)^{pk_2}.\\
\end{array}
$$
Combining the estimates above, we have
$$
\sup_{2r<|x|<3r}|v(x)|=
\sup_{2<|x|<3}|v_r(x)|\leq
Cc_0r^{2-k_1}(\ln r)^{k_2+1},
$$
where $C$ relies only on $k_1,k_2$ and $p$.
\end{proof}

By H\"older inequality, the constant $C$ relying on $p$ in \eqref{equ:FastVanishing}  remains finite when $p=\infty$ in  \eqref{equ:vanishing-g-Lp}.
For reading simplicity, hereinafter we let $v_g$ denote the solution constructed in Lemma \ref{Lem-fastConverge}.
Similar to Lemma 3.2 of \cite{bao-liu2020asymptotic}, vanishing speed of derivatives of $v_g$ follow immediately by Schauder estimates.
\begin{lemma}\label{Lem-fastConverge-2}
  Let $g\in C^{\infty}(\mathbb R^2)$ satisfy
  \begin{equation}\label{equ:vanishing-g}
  g=O_l(|x|^{-k_1}(\ln|x|)^{k_2})\quad\text{as }|x|\rightarrow\infty
  \end{equation}
  for some $k_1>2,k_2\geq 0$ and $l-1\in\mathbb N$. Then
  $$
  v_g=O_{l+1}(|x|^{2-k_1}(\ln |x|)^{k_2+1}).
  $$
\end{lemma}
Consequently, similar to Lemma 3.3, Proposition 3.4 and Corollary 3.5 of \cite{bao-liu2020asymptotic}, we have asymptotic expansion of solutions of \eqref{equ:Laplace} etc.
\begin{proposition}\label{prop:expansion}
  Let $v$ be a classical solution of
  \begin{equation}\label{equ:Poisson}
  a_{ij}(x)D_{ij}v=g(x)\quad\text{in }\mathbb R^2\setminus\overline{B_1}
  \end{equation}
  with $v$ satisfying \eqref{equ:asym-decomp},  $g$ satisfying \eqref{equ:vanishing-g} for some $k_1>3,k_2\geq 0$, $l-1\in\mathbb N$ and
  $$
  a_{ij}(x)-a_{ij}(\infty)=O_l(|x|^{-2})
  $$
  for some positive matrix $[a_{ij}(\infty)]$. Then
   $$
   Q=[a_{ij}(\infty)]^{-1}
   $$
  and there exist constants $d_1,d_2$ such that
  $$
  \begin{array}{llll}
  v&=&\displaystyle  d\ln x^TQx+(x^TQx)^{-\frac{1}{2}}\left(d_1 \cos \theta +
  d_2 \sin \theta\right)
  \\
  &&\displaystyle +\left\{
  \begin{array}{llll}
    O_{l+1}(|x|^{2-k_1}(\ln|x|)^{k_2+1}), & k_1\leq 4,\\
    O_{l+1}(|x|^{-2}(\ln|x|)), & k_1>4,\\
  \end{array}
  \right.\\
  \end{array}
  $$
  as $|x|\rightarrow\infty$, where $\theta(x)=\frac{Q^{\frac{1}{2}} x}{\left(x^{T} Q x\right)^{\frac{1}{2}}}$.
\end{proposition}
\begin{proof}
  Step 1. We prove the special case $a_{ij}(\infty)=\delta_{ij}$. Rewrite Eq.~ \eqref{equ:Poisson} into
  $$
  \Delta v=g(x)+(\delta_{ij}-a_{ij}(x))D_{ij}v=:\overline g(x)\quad\text{in }\mathbb R^2\setminus\overline{B_1}.
  $$
  By the assumptions on $g$, $v$ and $a_{ij}(x)$, we have
  $$
  \bar g(x)=O_l(|x|^{-k_1}(\ln|x|)^{k_2})+O_l(|x|^{-4})=\left\{
  \begin{array}{llll}
    O_l(|x|^{-k_1}(\ln|x|)^{k_2}), & k_1\leq 4,\\
    O_l(|x|^{-4}), & k_1>4,\\
  \end{array}
  \right.
  $$
  as $|x|\rightarrow\infty$. By Lemma \ref{Lem-fastConverge-2} and condition \eqref{equ:asym-decomp}, we have a solution $v_{\bar g}$ such that
  $$
  \Delta (v-v_{\bar g})=0\quad\text{in }\mathbb R^2\setminus\overline{B_1}
  $$
  with $$
  v-v_{\bar g}=d\ln(x^TQx)+O(|x|^{-1})\quad\text{and}\quad v_{\bar g}=\left\{
  \begin{array}{llll}
    O_{l+1}(|x|^{2-k_1}(\ln|x|)^{k_2+1}), & k_1\leq 4,\\
    O_{l+1}(|x|^{-2}(\ln|x|)), & k_1>4.\\
  \end{array}
  \right.
  $$
  By spherical harmonic expansion as in the proof of Lemma \ref{Lem-fastConverge}, there exist constants $C_{k,m}^{(1)}$, $C_{k,m}^{(2)}$ such that
  $$
  v-v_{\bar g}=
  C_{0,1}^{(1)}+C_{0,1}^{(2)}\ln|x|+
  \sum_{k=1}^{\infty}\sum_{m=1}^{2}C_{k,m}^{(1)}|x|^kY_m^{(k)}(\theta)
  +\sum_{k=1}^{\infty}\sum_{m=1}^{2}C_{k,m}^{(2)} |x|^{-k}Y_m^{(k)}(\theta).
  $$
  By the vanishing speed of $v-v_{\bar g}$ and linear independence of $Y_m^{(k)}(\theta)$, we have $Q=I$
  and $C_{k,m}^{(1)}=0$ for all $k, m$.
  This finishes the proof of this case.

  Step 2. We prove for general $[a_{ij}(\infty)]>0$. As in Lemma 6.1 of \cite{GT}, let $\widetilde Q:=[a_{ij}(\infty)]^{\frac{1}{2}}$ and $V(x):=v(\widetilde Qx)$.
  Since trace is invariant under cyclic permutations,
  $$
  \Delta V(x)=(a_{ij}(\infty)-a_{ij}(\widetilde Qx))D_{ij}v(\widetilde Qx)=:\bar g(x)
  $$
  in ${\widetilde Q}^{-1}(\mathbb R^2\setminus\overline{B_1})$. Then the result follows exactly the same as previous step.
\end{proof}

Eventually, we finish this section by proving Theorem \ref{thm:classify-constant}.
Applying Newton--Leibnitz formula between $F_{\tau}(\lambda(D^2u))=C_0$ and $F_{\tau}(\lambda(A))=C_0$, we have
$$
a_{ij}(x)D_{ij}\left(u-\left(\frac{1}{2} x^{T} A x+\beta x+\gamma\right) \right)=0\quad\text{in }\mathbb R^2\setminus\overline{\Omega},
$$
for some bounded domain $\Omega\subset\mathbb R^2$,
where $F(M):=F_{\tau}(\lambda(M))$ and
$$
a_{ij}(x)=\int_0^1D_{M_{ij}}F(A+t(D^2u-A))dt\rightarrow
a_{ij}(\infty)=D_{M_{ij}}F(A)
$$
as $|x|\rightarrow\infty$. By \eqref{equ:asym-Behavior}, we have
$$
a_{ij}(x)-a_{ij}(\infty)=O_k(|x|^{-2})
$$
for all $k\in\mathbb N$. By Proposition \ref{prop:expansion}, the desire result follows immediately.

Remark \ref{rem:final} follows similarly.

\section{Proof of Theorem \ref{Thm:classify-n=2}}\label{sec-symmetric}

The proof of symmetry is very similar to higher dimension case. Similar to the proof in \cite{bao-liu2021symmetry}, by Proposition 2.1 of \cite{Jin-Xiong} and Theorem 1.1 of \cite{CLN-Remarks-III}, $u$ is continuous and  a viscosity subsolution of $F_{\tau}(\lambda(D^2u))=C_0$ in $\mathbb R^2$.

Firstly, we consider the case when $\beta=0$ and $A$ is diagonal i.e., $A=\mathtt{diag}(a_1,a_2)$. Consider
\begin{equation*}
U(x):=u(\widetilde{x}), \quad \text { where }\widetilde x=(x_1,-x_2)\quad\text{or }(-x_1,x_2).
\end{equation*}
By a direct computation,   $u, U$ satisfy
$$
\left\{
\begin{array}{llll}
  F_{\tau}(\lambda(D^2u))=F_{\tau}(\lambda(D^2U))=C_0,& \text{in }\mathbb R^2\setminus\{0\},\\
  u(x)=U(x), & \text{at }x=0,\\
  u(x),U(x)\rightarrow\left( \frac{1}{2}x^TAx+\gamma+d\ln(x^TQx)\right),& \text{as }|x|\rightarrow\infty,
\end{array}
\right.
$$
where $Q$ is a polynomial of $A$ and hence also diagonal.
By Newton--Leibnitz formula between $u,U$ and applying maximum principle as in Lemma 2.4 of \cite{bao-liu2021symmetry}, we have $u(x)=U(x)=u(\widetilde x)$. Hence $u(x)$ is hyperplane symmetry in $x_1$ and $x_2$ directions.

Consequently, the symmetry \eqref{equ:symmetry} holds also for general matrix $A$. In fact, by eigen-decomposition  there exists an orthogonal matrix $O$ such that
\begin{equation}\label{equ:eigendecomposition-2}
A=O^T\Lambda O,\quad\text{where}\quad \Lambda=\mathtt{diag}(a_1,a_2).
\end{equation}
Since $O^TO=I$ and $Q$ is a polynomial of $A$, which is denoted by $P(A)$, we have $$
Q=O^T\Gamma O,\quad\text{where}\quad \Gamma=P(\Lambda).
$$
Let $v(x):=u(O^{-1}x)$, then $D^2v(x)=O^{-1}D^2u(O^{-1}x)O$ has the same eigenvalues as $D^2u(O^{-1}x)$. Hence
$$
F_{\tau}(\lambda(D^2v))=C_0\quad\text{in }\mathbb R^2\setminus\{0\},
$$
$v(0)=u(0)$ and
$$
v(x)\rightarrow \frac{1}{2}x^TO^{-1}AOx+\gamma+d\ln(x^TO^{-1}QOx)=\frac{1}{2}x^T\Lambda x+\gamma+d\ln(x^T\Gamma x)
$$
as $|x|\rightarrow\infty$. Then the symmetry follows from the result above.

\section{Proof of Remark \ref{Remark:value-of-d}}\label{sec-remark}

Rewrite asymptotic behavior \eqref{equ:asym-Behavior} into
$u=W(x)+\Gamma(x)+O_k(|x|^{-1}),$
where
$$
W(x):=\frac{1}{2} x^{T} A x+\beta x+\gamma,\quad
\Gamma(x):=d \ln \left(x^{T} Q x\right).
$$
The computation of $d$ relies on the fact that $F_{\tau}(\lambda(D^2u))=C_0$ and
$F_{\tau}(\lambda(D^2W))=C_0$.
Let
$$
E_{R}:=\{x\in\mathbb R^2|x^TQx<R^2\}
$$
and $W_{ij}, \Gamma_{ij}$ denote the partial derivatives $D_{ij}W$, $D_{ij}\Gamma$ for all $i,j=1,2$.
\begin{lemma}\label{Lem-fundamentalSol}
  For a positive matrix $Q>0$ and $d\in\mathbb R$, $\Gamma$ defined above satisfies
  \begin{equation}\label{equ:distribution}
    \mathtt{div} (Q^{-1}\nabla\Gamma(x))=\dfrac{4\pi d}{|Q^{\frac{1}{2}}|} \delta_0\quad\text{in }\mathbb R^2,
  \end{equation}
  in distribution sense,
  where $\delta_0$ denotes the Dirac mass centered at origin.
\end{lemma}
\begin{proof}
  Change of variable by setting $y:=Q^{\frac{1}{2}}x$, then
  $$
  \Gamma(x(y))=d\ln(x^TQx)=d\ln|y|^2.
  $$
On the one hand,
$$
\Delta_y\Gamma(x(y))=4\pi d \delta_0\quad\text{in }\mathbb R^2\setminus\{0\}.
$$
On the other hand, by  the equation satisfied by $\Gamma$, for any $\varphi\in C_c^{\infty}(\mathbb R^2)$ we have
$$
\int_{\mathbb R^2}\nabla_y \Gamma(x(y))\cdot \nabla_y \varphi(y) dy=4\pi d \varphi(0)
$$
and consequently
$$
\int_{\mathbb R^2}\nabla_x \Gamma(x) Q^{-1}\nabla_x \varphi(y(x))\cdot |Q^{\frac{1}{2}}|dx=4\pi d \varphi(0).
$$
Integral by parts and the result follows immediately.
\end{proof}

For $\tau=0$ case, the proof can be found in \cite{CL,RemarkMA-2020,Li-Li-Yuan-Bernstein-SPL} etc. For reading simplicity and further usage, we include a proof here.
Integral on $E_{R}\setminus\overline{\Omega'}$ for sufficiently large $R$, we have
$$
e^{2C_0}|E_R\setminus\overline{\Omega'}|=
\int_{E_R\setminus\overline{\Omega'}}\det D^2u dx=\int_{\partial(E_R\setminus\overline{\Omega'})}u_1(u_{22},-u_{12})\cdot\nu ds.
$$
By the asymptotic behavior,
$$
\begin{array}{lllll}
& \displaystyle \int_{\partial E_R}u_1(u_{22},-u_{12})\cdot\nu ds\\
=& \displaystyle \int_{\partial E_R} (W_1+\Gamma_1)(W_{22}+\Gamma_{22},-W_{12}-\Gamma_{12})\cdot\nu ds+O(R^{-1})\\
=& \displaystyle \int_{\partial E_R} W_1(W_{22},-W_{12})\cdot\nu ds\\
&\displaystyle +\int_{\partial E_R}W_1(\Gamma_{22},-\Gamma_{12})\cdot \nu ds
+\int_{\partial E_R}\Gamma_1(W_{22},-W_{12})\cdot \nu ds+O(R^{-1}).\\
\end{array}
$$
Integral by parts and we have
$$
\int_{\partial E_R}W_1(\Gamma_{22},-\Gamma_{12})\cdot \nu ds
=
\int_{E_R}(W_{11}\Gamma_{22}-W_{12}\Gamma_{12})dx,
$$
$$
\int_{\partial E_R}\Gamma_1(W_{22},-W_{12})\cdot \nu ds
=
\int_{E_R}(W_{22}\Gamma_{11}-W_{12}\Gamma_{12})dx.
$$
For simplicity, we assume $A$ is diagonal with eigenvalues $a_1,a_2$. Then we have
$Q=2A$. By Lemma \ref{Lem-fundamentalSol}, we have
$$
(2a_1)^{-1}\Gamma_{11}+(2a_2)^{-1}\Gamma_{22}=\dfrac{4\pi d }{(4a_1a_2)^{\frac{1}{2}}}\delta_0\quad\text{in }\mathbb R^2.
$$
Consequently,
$$
\begin{array}{llll}
&\displaystyle \int_{\partial E_R}\left(W_{1}\left(\Gamma_{22},-\Gamma_{12}\right)
+\Gamma_{1}\left(W_{22},-W_{12}\right)
\right)
\cdot \nu d s\\
=&\displaystyle \int_{E_R }(a_1\Gamma_{22}+a_2\Gamma_{11})dx\\
=&\displaystyle  4\pi d (a_1a_2)^{\frac{1}{2}}
= 4\pi d e^{C_0}.\\
\end{array}
$$
By the calculus above,
$$
e^{2C_0}|E_R\setminus\overline{\Omega'}|=e^{2C_0}|E_R|-\int_{\partial\Omega'}
u_1(u_{22},-u_{12})\cdot\nu ds
+4\pi d e^{C_0}+O(R^{-1}).
$$
Sending $R$ to infinity, \eqref{equ:represent-d-MA} follows immediately.

For $\tau\in (0,\frac{\pi}{4})$ case, Eq.~\eqref{equ:Perturbed-General} can be represented into
$$
\det (D^2(u+\frac{a-b}{2}|x|^2))=c_0\det (D^2(u+\frac{a+b}{2}|x|^2)),
$$
where $c_0:=\exp(\frac{2b}{\sqrt{a^2+1}}C_0)<1.$ By a direct computation, $W$ satisfies the equation above as well. Integral over $E_R\setminus\overline{\Omega'}$ for sufficiently large $R$, we have
$$
\begin{array}{lllll}
0&=&\displaystyle \int_{E_R\setminus\overline{\Omega'}}\left(\det (D^2(u+\frac{a-b}{2}|x|^2))-c_0\det (D^2(u+\frac{a+b}{2}|x|^2))\right)dx\\
&=& \displaystyle \int_{\partial (E_R\setminus\overline{\Omega'})}
(u_1+(a-b)x_1)(u_{22}+a-b,-u_{12})\cdot \nu ds\\
&&\displaystyle -c_0\int_{\partial (E_R\setminus\overline{\Omega'})}
(u_1+(a+b)x_1)(u_{22}+a+b,-u_{12})\cdot \nu ds.\\
\end{array}
$$
By the asymptotic behavior,
$$
\begin{array}{lllll}
  &\displaystyle \int_{\partial E_R}
(u_1+(a-b)x_1)(u_{22}+a-b,-u_{12})\cdot \nu ds\\
=&\displaystyle \int_{\partial E_R}
(W_1+\Gamma_1+(a-b)x_1)(W_{22}+\Gamma_{22}+a-b,-W_{12}-\Gamma_{12})\cdot \nu ds+O(R^{-1})\\
=&\displaystyle  \int_{\partial E_R}
(W_1+ (a-b)x_1)(W_{22}+a-b,-W_{12})\cdot \nu ds\\
&\displaystyle +\int_{\partial E_R}\Gamma_1(W_{22}+a-b,-W_{12})\cdot\nu ds
+\int_{\partial E_R}(W_1+(a-b)x_1)(\Gamma_{22},-\Gamma_{12})\cdot\nu ds +O(R^{-1}).\\
\end{array}
$$
Integral by parts and we have
$$
\int_{\partial E_R}\Gamma_1(W_{22}+a-b,-W_{12})\cdot\nu ds=
\int_{E_R}\left(\Gamma_{11}(W_{22}+(a-b))-\Gamma_{12}W_{12}\right)dx,
$$
$$
\int_{\partial E_R}(W_1+(a-b)x_1)(\Gamma_{22},-\Gamma_{12})\cdot\nu ds =
\int_{E_R}\left(\Gamma_{22}(W_{11}+(a-b))-\Gamma_{12}W_{12}\right)dx.
$$
Similarly,
$$
\begin{array}{lllll}
  &\displaystyle \int_{\partial E_R}
(u_1+(a+b)x_1)(u_{22}+a+b,-u_{12})\cdot \nu ds\\
=&\displaystyle \int_{\partial E_R}
(W_1+\Gamma_1+(a+b)x_1)(W_{22}+\Gamma_{22}+a+b,-W_{12}-\Gamma_{12})\cdot \nu ds+O(R^{-1})\\
=&\displaystyle  \int_{\partial E_R}
(W_1+ (a+b)x_1)(W_{22}+a+b,-W_{12})\cdot \nu ds\\
&\displaystyle +\int_{\partial E_R}\Gamma_1(W_{22}+a+b,-W_{12})\cdot\nu ds
+\int_{\partial E_R}(W_1+(a+b)x_1)(\Gamma_{22},-\Gamma_{12})\cdot\nu ds +O(R^{-1}).\\
\end{array}
$$
Similar to $\tau=0$ case, we consider the case where $A$ is diagonal with eigenvalues $a_1,a_2$. Then we have
$$
Q=\mathtt{diag}(\sin\tau a_1^2+2\cos\tau a_1+\sin\tau, \sin\tau a_2^2+2\cos\tau a_2+\sin\tau).
$$
By Lemma \ref{Lem-fundamentalSol} we have
$$
\begin{array}{llll}
 &(a_2+a-b)(a_2+a+b)\Gamma_{11}+(a_1+a-b)(a_1+a+b)\Gamma_{22}\\
=&4\pi d \sqrt{(a_1+a-b)(a_1+a+b)(a_2+a-b)(a_2+a+b)} \delta_0
\end{array}
\quad\text{in }\mathbb R^2
$$
and hence
$$
\begin{array}{lllll}
  &\displaystyle \int_{\partial E_R}\Gamma_1(W_{22}+a-b,-W_{12})\cdot\nu ds
+\int_{\partial E_R}(W_1+(a-b)x_1)(\Gamma_{22},-\Gamma_{12})\cdot\nu ds \\
&\displaystyle -c_0
\int_{\partial E_R}\Gamma_1(W_{22}+a+b,-W_{12})\cdot\nu ds
-c_0\int_{\partial E_R}(W_1+(a+b)x_1)(\Gamma_{22},-\Gamma_{12})\cdot\nu ds \\
=& \displaystyle \int_{E_R}\left(\Gamma_{11}(W_{22}+a-b)+\Gamma_{22}(W_{11}+a-b)\right)dx\\
&\displaystyle -c_0\int_{E_R}\left(\Gamma_{11}(W_{22}+a+b)+\Gamma_{22}(W_{11}+a+b)\right)dx\\
=& \displaystyle\dfrac{2b}{(a_1+a+b)(a_2+a+b)}
\int_{E_R}\left((a_2+a-b)(a_2+a+b)\Gamma_{11}+(a_1+a-b)(a_1+a+b)\Gamma_{22}\right)dx\\
=& \displaystyle 8\pi d  \left(\dfrac{(a_1+a-b)(a_2+a-b)}{(a_1+a+b)(a_2+a+b)}\right)^{\frac{1}{2}}\\
=& 8\pi d c_0^{\frac{1}{2}}.
\end{array}
$$
By the calculus above,
$$
\begin{array}{llllll}
  & 8\pi dc_0^{\frac{1}{2}}+O(R^{-1}) \\
=& \displaystyle \int_{\partial\Omega'}(u_1+(a-b)x_1)(u_{22}+a-b,-u_{12})\cdot \nu ds\\
&\displaystyle -c_0\int_{\partial\Omega'}(u_1+(a+b)x_1)(u_{22}+a+b,-u_{12})\cdot \nu ds.\\
\end{array}
$$
Sending $R$ to infinity, \eqref{equ:represent-d-small} follows immediately.

For $\tau=\frac{\pi}{4}$ case, Eq.~\eqref{equ:Perturbed-General} is a translated inverse harmonic hessian equation, whose asymptotic behavior is proved in Li--Li--Yuan \cite{Li-Li-Yuan-Bernstein-SPL}.
Since there are slightly different on the coefficients, we provide the proof here. Rewrite the equation into
$$
\Delta u+2+c_0\det (D^2(u+\frac{1}{2}|x|^2))=0,
$$
where $c_0=\frac{\sqrt 2}{2}C_0$.
By a direct computation, $W$ satisfies the equation above as well. Integral over $E_{R} \backslash \overline{\Omega^{\prime}}$ for sufficiently large $R$, we have
$$
\begin{array}{llll}
  -2|E_R\setminus\overline{\Omega'}| &= & \displaystyle
  \int_{E_R\setminus\overline{\Omega'}}\left(\Delta u+c_0\det \left(D^2\left(u+\frac{1}{2}|x|^2\right)\right)\right)d x\\
  &=&\displaystyle
  \int_{\partial (E_R\setminus\overline{\Omega'})}\left((u_1,u_2)\cdot \nu+c_0 (u_1+x_1)(u_{22}+1,-u_{12})\cdot \nu\right) ds.\\
\end{array}
$$
By the asymptotic behavior,
$$
\begin{array}{llll}
  &\displaystyle \int_{\partial E_R}(u_1,u_2)\cdot \nu+c_0 (u_1+x_1)(u_{22}+1,-u_{12})\cdot \nu ds\\
  =& \displaystyle \int_{\partial E_R}(W_1+\Gamma_1,W_2+\Gamma_2)\cdot \nu ds\\
  &\displaystyle +c_0
  \int_{\partial E_R} (W_1+\Gamma_1+x_1)(W_{22}+\Gamma_{22}+1,-W_{12}-\Gamma_{12})\cdot \nu ds+O(R^{-1})\\
  =&\displaystyle \int_{\partial E_R}(W_1,W_2)\cdot \nu+c_0 (W_1+x_1)(W_{22}+1,-W_{12})\cdot \nu ds\\
  &+\displaystyle \int_{\partial E_R}(\Gamma_1,\Gamma_2)\cdot \nu ds
  +c_0\int_{\partial E_R}\Gamma_1(W_{22}+1,-W_{12})\cdot \nu ds
  \\
  &+\displaystyle c_0\int_{\partial E_R}(W_1+x_1)(\Gamma_{22},-\Gamma_{12})\cdot \nu ds+O(R^{-1}).\\
\end{array}
$$
Integral by parts and we have
$$
\int_{\partial E_R} (\Gamma_1,\Gamma_2)\cdot \nu ds=\int_{E_R}(\Gamma_{11} +\Gamma_{22})dx,
$$
$$
\int_{\partial E_{R}} \Gamma_{1}\left(W_{22}+1,-W_{12}\right) \cdot \nu d s=
\int_{E_R}
\left(\Gamma_{11}(W_{22}+1)-\Gamma_{12}W_{12}\right)dx,
$$
$$
\int_{\partial E_R}(W_1+x_1)(\Gamma_{22},-\Gamma_{12})\cdot \nu ds=\int_{E_R}(\Gamma_{22}(W_{11}+1)-\Gamma_{12}W_{12})dx.
$$
Similar to $\tau=0$ case, we consider the case where $A$ is diagonal with eigenvalues $a_{1}, a_{2}$ satisfying $$
\dfrac{1}{a_1+1}+\dfrac{1}{a_2+1}=-c_0.
$$
Then we have
$$
Q=\mathtt{diag}\left(\frac{1}{2}(a_1+1)^2,\frac{1}{2}(a_2+1)^2\right).
$$
By Lemma \ref{Lem-fundamentalSol}, we have
$$
(a_2+1)^2\Gamma_{11}+(a_1+1)^2\Gamma_{22}=4\pi d \delta_0 \cdot (a_1+1)(a_2+1)
$$
in $\mathbb R^2$
and hence
$$
\begin{array}{lllll}
  &\displaystyle \int_{\partial E_R}\left(\Gamma_{1}, \Gamma_{2}\right)\cdot \nu ds
  +c_0\int_{\partial E_R}\Gamma_1(W_{22}+1,-W_{12})\cdot \nu ds
  \\
  &+\displaystyle c_0\int_{\partial E_R}(W_1+x_1)(\Gamma_{22},-\Gamma_{12})\cdot \nu ds\\
  =& \displaystyle \int_{E_R} (\Gamma_{11}+\Gamma_{22})d x-\dfrac{a_1+a_2+2}{(a_1+1)(a_2+1)}\int_{E_R}(\Gamma_{11}(a_2+1)+\Gamma_{22}(a_1+1))dx\\
  =&\displaystyle
  \int_{E_R} -\dfrac{\left(a_{2}+1\right)^{2} \Gamma_{11}+\left(a_{1}+1\right)^{2} \Gamma_{22}}{(a_1+1)(a_2+1)}dx\\
  =& \displaystyle  -4\pi d  .\\
\end{array}
$$
By the calculus above,
$$
\begin{array}{llllll}
 & -2\left|E_{R} \backslash \overline{\Omega^{\prime}}\right|\\
=& \displaystyle -2|E_R|-\int_{\partial\Omega'}\left(\left(u_{1}, u_{2}\right) \cdot \nu+c_{0}\left(u_{1}+x_{1}\right)\left(u_{22}+1,-u_{12}\right) \cdot \nu\right) d s-4\pi d.
\end{array}
$$
Sending $R$ to infinity, \eqref{equ:represent-d-mid} follows immediately.

For $\tau\in(\frac{\pi}{4},\frac{\pi}{2})$, the proof of \eqref{equ:represent-d-large} follows from $\tau=\frac{\pi}{2}$ case by identity \eqref{equ:identity}. For $\tau=\frac{\pi}{2}$, the proof can be found in Li--Li--Yuan \cite{Li-Li-Yuan-Bernstein-SPL}. The major technique here is to rewrite the equation into linear combination of $\Delta u$ and $\det D^2u$ as in
$$
\cos C_0\Delta u+\sin C_0\det D^2u=\sin C_0.
$$
Then following the same calculus above and Lemma \ref{Lem-fundamentalSol}, the result follows immediately.

\small

\bibliographystyle{plain}

\bibliography{AsymExpan-Dim2}

\bigskip

\noindent Z.Liu \& J. Bao

\medskip

\noindent  School of Mathematical Sciences, Beijing Normal University\\
Laboratory of Mathematics and Complex Systems, Ministry of Education\\
Beijing 100875, China \\[1mm]
Email: \textsf{liuzixiao@mail.bnu.edu.cn, jgbao@bnu.edu.cn}

\end{document}